\documentclass[11pt]{article}
\usepackage{amssymb}
\newtheorem{precor}{{\bf Corollary}}

\newtheorem{precon}{{\bf Conjecture}}

\newtheorem{prealphcon}{{\bf Conjecture}}

\newtheorem{predefin}{{\bf Definition}}

\newtheorem{preexm}{{\bf Example}}

\newtheorem{preappl}{{\bf Application}}

\newtheorem{prelem}{{\bf Lemma}}

\newtheorem{preproof}{{\bf Proof.\ }}

\newenvironment{proof}[1]{\begin{preproof}{\rm
               #1}\hfill{$\blacksquare$}}{\end{preproof}}
\newtheorem{pretheorem}{{\bf Theorem}}

\newenvironment{theorem}{\begin{pretheorem}{\hspace{-0.5
               em}{\bf.\ }}}{\end{pretheorem}}
\newtheorem{prealphtheorem}{{\bf Theorem}}

\newtheorem{prealphlem}{{\bf Lemma}}

\newtheorem{prepro}{{\bf Proposition}}

\newtheorem{preprb}{{\bf Problem}}

\newtheorem{prerem}{{\bf Remark}}

\newtheorem{preapp}{{\bf Application}}

\newtheorem{prequ}{{\bf Question}}

\def\conct[#1,#2]{\mbox {${#1} \leftrightarrow {#2}$}}
\def\dconct[#1,#2]{\mbox {${#1} \rightarrow {#2}$}}
\def\deg[#1,#2]{\mbox {$d_{_{#1}}(#2)$}}
\def\mindeg[#1]{\mbox {$\delta_{_{#1}}$}}
\def\maxdeg[#1]{\mbox {$\Delta_{_{#1}}$}}
\def\outdeg[#1,#2]{\mbox {$d_{_{#1}}^{^+}(#2)$}}
\def\minoutdeg[#1]{\mbox {$\delta_{_{#1}}^{^+}$}}
\def\maxoutdeg[#1]{\mbox {$\Delta_{_{#1}}^{^+}$}}
\def\indeg[#1,#2]{\mbox {$d_{_{#1}}^{^-}(#2)$}}
\def\minindeg[#1]{\mbox {$\delta_{_{#1}}^{^-}$}}
\def\maxindeg[#1]{\mbox {$\Delta_{_{#1}}^{^-}$}}

\def\dre[#1,#2,#3]{\mbox {${\cal E}^{^{#3}}(#1,#2)$}}
\def\var[#1,#2]{\mbox {${\rm Var}_{_{#1}}(#2)$}}
\def\ls[#1]{\mbox {$\xi^{^{#1}}$}}
\def\hom[#1,#2]{\mbox {${\rm Hom}({#1},{#2})$}}
\def\onvhom[#1,#2]{\mbox {${\rm Hom^{v}}(#1,#2)$}}
\def\onehom[#1,#2]{\mbox {${\rm Hom^{e}}(#1,#2)$}}
\def\core[#1]{\mbox {$#1^{^{\bullet}}$}}
\def\cay[#1,#2]{\mbox {${\rm Cay}({#1},{#2})$}}
\def\sch[#1,#2,#3]{\mbox {${\rm Sch}({#1},{#2},{#3})$}}
\def\cays[#1,#2]{\mbox {${\rm Cay_{s}}({#1},{#2})$}}
\def\dirc[#1]{\mbox {$\stackrel{\rightarrow}{C}_{_{#1}}$}}
\def\cycl[#1]{\mbox {${\bf Z}_{_{#1}}$}}

\begin{document}

\begin{center} 
{\Large \bf On Local Antimagic Chromatic Number of Graphs}\\
\vspace{0.3 cm}
{\bf Saeed Shaebani}\\
{\it School of Mathematics and Computer Science}\\
{\it Damghan University}\\
{\it P.O. Box {\rm 36716-41167}, Damghan, Iran}\\
{\tt shaebani@du.ac.ir}\\ \ \\
\end{center}
\begin{abstract}
\noindent A {\it local antimagic labeling} of a connected graph $G$ with at least
three vertices, is a bijection $f:E(G) \rightarrow \{1,2,\ldots , |E(G)|\}$ such
that for any two adjacent vertices $u$ and $v$ of $G$, the condition 
$\omega _{f}(u) \neq \omega _{f}(v)$ holds; where
$\omega _{f}(u)=\sum _{x\in N(u)} f(xu)$. Assigning $\omega _{f}(u)$ to
$u$ for each vertex $u$ in $V(G)$, induces naturally a proper vertex coloring
of $G$; and $|f|$ denotes the number of colors appearing in this proper vertex coloring.
The {\it local antimagic chromatic number} of $G$, denoted by $\chi _{la}(G)$, is defined
as the minimum of $|f|$, where $f$ ranges over all local antimagic labelings of $G$.
In this paper, we explicitely construct an infinite class 
of connected graphs $G$ such that $\chi _{la}(G)$
can be arbitrarily large while $\chi _{la}(G \vee \bar{K_{2}})=3$, where $G \vee \bar{K_{2}}$
is the join graph of $G$ and the complement graph of $K_{2}$. This fact leads to a
counterexample to a theorem of [Local antimagic vertex 
coloring of a graph, {\em Graphs and Combinatorics}\ {\bf
33} (2017), 275--285].
\\


\end{abstract}
\section{Introduction}

Unless otherwise stated we consider connected finite simple graphs that have at least
three vertices. Let $G$ be a graph and $f:E(G) \rightarrow \{1,2,\ldots , |E(G)|\}$ be a bijection.
For each vertex $u$ in $V(G)$, we mean by $\omega _{f} (u)$ as the sum of the labels
of all incident edges to $u$; more precisely, $\omega _{f}(u)=\sum _{x\in N(u)} f(xu)$.
Whenever there is no ambiguity on $f$, we use 
the symbol $\omega (u)$ instead of $\omega _{f}(u)$.

Let $G$ be a graph and $f:E(G) \rightarrow \{1,2,\ldots , |E(G)|\}$ be a bijection.
If $\omega _{f} (u)  \neq \omega _{f} (v)$ for any two distinct vertices $u$ and $v$ in $V(G)$,
then $f$ is called an {\it antimagic labeling} of $G$ \cite{Harts}. 
Hartsfield and Ringel  conjectured that every connected graph with at
least three vertices admits an antimagic labeling \cite{Harts}. By several authors, this conjecture was
shown to be true for some special classes of graphs, but it is still widely unsolved.
The important fact about this conjecture is that it is unsolved even for
trees; see \cite{Bens} for an interesting discussion in this topic.

In 2017, Arumugam, Premalatha, Ba{\v{c}}a, and 
Semani{\v{c}}ov{\'a}-Fe{\v{n}}ov{\v{c}}{\'i}kov{\'a} in  \cite{Aru}, and independently, 
Bensmail, Senhaji, and Lyngsie in \cite{Bens}, posed a new definition as a relaxation of
the notion of antimagic labeling. They called a bijection 
$f:E(G) \rightarrow \{1,2,\ldots , |E(G)|\}$
a {\it local antimagic labeling} of $G$ if 
for any two adjacent vertices $u$ and $v$ in $V(G)$, the condition
$\omega_{f} (u) \neq \omega_{f} (v)$ holds. They conjectured that every 
connected graph with at least three vertices admits a local antimagic labeling.
This conjecture was solved partially in \cite{Bens}. A few months later, 
Haslegrave proved this conjecture by means of probabilistic tools \cite{Hasl}.

Based on the notion of local antimagic labeling,
Arumugam, Premalatha, Ba{\v{c}}a, and 
Semani{\v{c}}ov{\'a}-Fe{\v{n}}ov{\v{c}}{\'i}kov{\'a}
introduced a new graph coloring parameter.
Let $G$ be a connected graph with at least three vertices, and 
$f:E(G) \rightarrow \{1,2,\ldots , |E(G)|\}$
be a local antimagic labeling of $G$. For any two adjacent vertices 
$u$ and $v$ we have
$\omega_{f} (u) \neq \omega_{f} (v)$; so, 
assigning $\omega _{f}(u)$ to
$u$ for each vertex $u$ in $V(G)$, induces naturally a proper vertex coloring
of $G$ which is called a 
{\it local antimagic vertex coloring} of $G$. 
Let $|f|$ denote the number of colors appearing in this proper vertex coloring.
More precisely, $|f|=|\{\omega_{f} (u) : u\in V(G)\}|$.
The {\it local antimagic chromatic number} of $G$, denoted by $\chi _{la}(G)$, is defined
as the minimum of $|f|$, where $f$ ranges over 
all local antimagic labelings of $G$  \cite{Aru}.

Let $G_{1}$ and $G_{2}$ be two vertex disjoint graphs.
The {\it join graph} of $G_{1}$ and $G_{2}$, denoted by
$G_{1} \vee G_{2}$, is the graph whose vertex set is
$V(G_{1}) \cup V(G_{2})$ and its edge set equals
$E(G_{1}) \cup E(G_{2}) \cup \{ab :\ a\in V(G_{1})\ {\rm and}\ b\in V(G_{2})\}$.

The Theorem $2.16$ of \cite{Aru} asserts that if a graph $G$
has at least four vertices, then
$\chi _{la}(G) + 1 \leq \chi _{la}(G \vee \bar{K_{2}})$, where $\bar{K_{2}}$
is the complement graph of a complete graph with two vertices.
In this paper, we show that the mentioned theorem is incorrect.
In this regard, we explicitely construct an infinite class 
of connected graphs $G$ such that $\chi _{la}(G)$
can be arbitrarily large and $\chi _{la}(G \vee \bar{K_{2}})=3$.

\section{The main result}

This section is devoted to construct an infinite class 
of connected graphs $G$ such that $\chi _{la}(G)$
can be arbitrarily large while $\chi _{la}(G \vee \bar{K_{2}})=3$.
Our procedure is to consider the complete bipartite graph $K_{1,n}$ that satisfies
$\chi _{la}(K_{1,n})=n+1$ for each positive integer $n \geq 2$. We show that
if $n$ is odd and $n+1$ is not divisible by $3$, then
$\chi _{la}(K_{1,n} \vee \bar{K_{2}})=3$.

\begin{theorem}{Let $n$ be an odd integer such that $n+1$ is not divisible by $3$.
Then, the join of the star graph $K_{1,n}$ and the complement of $K_{2}$, say
$H:=K_{1,n} \vee \bar{K_{2}}$, satisfies $\chi _{la}(H)=3$.
}
\end{theorem}

\begin{proof}{
Let the vertex set of the star graph $K_{1,n}$ be $\{v,v_{1},v_{2}, \ldots ,v_{n}\}$ and $v$
be its central vertex. Also, let $x$ and $y$ be the two vertices of $\bar{K_{2}}$.
Since $H$ has some triangles, we have $\chi _{la}(H) \geq \chi (H) \geq 3$.
So, for proving $\chi _{la}(H)=3$, it suffices to provide a local antimagic labeling
of $H$ that induces a local antimagic vertex coloring using exactly three colors.

For $n=1$, define $f: E(H) \rightarrow \{1,2,3,4,5\}$ by
$$\begin{array}{lcr}
 f(vv_{1})=1,  &  f(vx)=5, & f(vy)=4, \\
 f(v_{1}x)=2,  &  f(v_{1}y)=3. & 
\end{array}$$
In this case, we have
$$\begin{array}{lcr}
 \omega (v)=10,  &  \omega (v_{1})=6, & \omega (x)=\omega (y)=7. 
\end{array}$$
Therefore, $f$ is a local antimagic labeling
of $H$ that induces a local antimagic vertex coloring using exactly three colors.

For $n \geq 3$, the aim is to construct a local antimagic labeling
$f: E(H) \rightarrow \{1,2,3, \ldots , 3n+2\}$ such that
$\omega (v_{1})=\omega (v_{2})=\cdots =\omega (v_{n})$
and $\omega (x)=\omega (y)$.
In this regard, we first assign $f(vv_{i})=i$ for each $i$ in $\{1,2,\ldots ,n\}$.
Also, in our construction, $\{f(vx),f(vy)\}=\{n+1,n+2\}$.
Therefore,
\begin{center}
$\omega (v)=\sum\limits _{i=1}^{n+2} i=\frac{(n+2)(n+3)}{2}.$
\end{center}
Also, we must have
\begin{center}
$\omega (x)=\omega (y)=\frac{1}{2}\sum\limits _{i=n+1}^{3n+2} i=\frac{(n+1)(4n+3)}{2},$
\end{center}
and
\begin{center}
$\omega (v_{1})=\omega (v_{2})=\cdots =\omega (v_{n})=\frac{9n+11}{2}.$
\end{center}
This shows that since $n+1$ is not divisible by $3$, the desired $f$
will be a local antimagic labeling
of $H$ and it induces a local antimagic vertex coloring of $H$ with three colors.
We make a partition $\{A_{1},A_{2},\ldots, A_{n}\}$ of the set
$\{n+3,n+4,\ldots, 3n+2\}$ such that for each $i$ in 
$\{1,2,\ldots ,n\}$, the set $A_{i}$ has two elements and
$A_{i}=\{f(v_{i}x),f(v_{i}y)\}$.
Also, $A_{i}$ has one element in 
$\{n+3,n+4,\ldots, 2n+2\}$ and one element in 
$\{2n+3,2n+4,\ldots, 3n+2\}$. In this regard, our suitable partition is as
the following;
$$A_{i}=\left\{\begin{array}{lcrc}
 \{2n+4-2i\ ,\ \frac{5n+3}{2}+i\}  &    &   \ {\rm if} \ & 1\leq i \leq \frac{n+1}{2} \\
   &    &  & \\
 \{3n+4-2i\ ,\ \frac{3n+3}{2}+i\}  &    &    \ {\rm if} \ &  \frac{n+3}{2} \leq i \leq n. 
\end{array}\right.$$
It is obvious that for each $i$ in $\{1,2,\ldots ,n\}$, we have
\begin{center}
$\omega (v_{i})= i+f(v_{i}x)+f(v_{i}y)=\frac{9n+11}{2}$.
\end{center}
Accordingly, the following $n+1$ sets
\begin{center}
$\{f(vx),f(vy)\},\{f(v_{1}x),f(v_{1}y)\},\ldots ,\{f(v_{n}x),f(v_{n}y)\}$
\end{center}
are determined.
For completing the proof, it is sufficient to determine the exact values of each of
$f(vx),f(vy),f(v_{1}x),f(v_{1}y),\ldots ,f(v_{n}x),f(v_{n}y)$, 
in such a way that
$\omega (x)=\omega (y)$. In this regard, we consider the following four cases.
\\

\noindent \textbf{Case 1.} The case that $\frac{n+1}{2}\stackrel{4}{\equiv} 0$.

\noindent First we determine $f(v_{i}x)$ and $f(v_{i}y)$ for each $i$ in $\{1,2,\ldots ,\frac{n+1}{2}\}$; as follows.

$$f(v_{i}x)=\left\{\begin{array}{ccrccccc}
 \frac{5n+3}{2}+i  &    &   \ {\rm if} \ &   1\leq i \leq \frac{n+1}{2}  &  {\rm and}    &   (i\stackrel{4}{\equiv} 1 & {\rm or} & i\stackrel{4}{\equiv} 0)\\
   &    &  &      &  &    &  & \\
 2n+4-2i  &    &   \ {\rm if} \ &   1\leq i \leq \frac{n+1}{2}  &   {\rm and}   &  (i\stackrel{4}{\equiv} 2 & {\rm or} & i\stackrel{4}{\equiv} 3)\\
\end{array}\right.$$
and
$$f(v_{i}y)=\left\{\begin{array}{ccrccccc}
 2n+4-2i  &    &   \ {\rm if} \ &   1\leq i \leq \frac{n+1}{2}  &   {\rm and}   & (i\stackrel{4}{\equiv} 1 & {\rm or} & i\stackrel{4}{\equiv} 0)\\ 
   &    &  &        &   &   &  & \\
   \frac{5n+3}{2}+i  &    &   \ {\rm if} \ &   1\leq i \leq \frac{n+1}{2}  &   {\rm and}   & (i\stackrel{4}{\equiv} 2 & {\rm or} & i\stackrel{4}{\equiv} 3).\\
\end{array}\right.$$

\noindent If $i$ is a positive integer such that $i\stackrel{4}{\equiv} 1$ and $i+4 \leq \frac{n+1}{2}$,
then
\begin{center}
$\sum\limits_{j=i}^{i+3} f(v_{j}x)=9n+8-2i$
\end{center}
and
\begin{center}
$\sum\limits_{j=i}^{i+3} f(v_{j}y)=9n+8-2i$.
\end{center}
This shows that since $\frac{n+1}{2}$ is divisible by $4$, we have
\begin{center}
$\sum\limits_{j=1}^{\frac{n+1}{2}} f(v_{j}x)=\sum\limits_{j=1}^{\frac{n+1}{2}} f(v_{j}y)$.
\end{center}

\noindent Now, we put $f(vx)=n+2$ and $f(vy)=n+1$. Also, for each $i$ in
$\left\{\frac{n+3}{2},\frac{n+5}{2},\frac{n+7}{2}\right\}$
put
$$f(v_{i}x)=\left\{\begin{array}{ccrc}
 \frac{3n+3}{2}+i  &    &   \ {\rm if} \ &   i \in \left\{\frac{n+3}{2},\frac{n+5}{2}\right\}\\
   &    &  & \\
 3n+4-2i  &    &   \ {\rm if} \ &   i = \frac{n+7}{2}  \\
\end{array}\right.$$
and
$$f(v_{i}y)=\left\{\begin{array}{ccrc}
 3n+4-2i  &    &   \ {\rm if} \ &   i \in \left\{\frac{n+3}{2},\frac{n+5}{2}\right\}\\
   &    &  & \\
 \frac{3n+3}{2}+i  &    &   \ {\rm if} \ &   i = \frac{n+7}{2}.  \\
\end{array}\right.$$
We have
\begin{center}
$f(vx)+\sum\limits_{j=\frac{n+3}{2}}^{\frac{n+7}{2}} f(v_{j}x)=f(vy)+\sum\limits_{j=\frac{n+3}{2}}^{\frac{n+7}{2}} f(v_{j}y).$
\end{center}
Therefore, 
\begin{center}
$f(vx)+\sum\limits_{j=1}^{\frac{n+7}{2}} f(v_{j}x)=f(vy)+\sum\limits_{j=1}^{\frac{n+7}{2}} f(v_{j}y).$
\end{center}
Now, it is turn to determine the exact values of
\begin{center}
$f\left(xv_{\frac{n+9}{2}}\right),f\left(yv_{\frac{n+9}{2}}\right),
f\left(xv_{\frac{n+11}{2}}\right),f\left(yv_{\frac{n+11}{2}}\right),
\ldots ,f(xv_{n}),f(yv_{n})$.
\end{center}
Consider the following assignments;

$$f(v_{i}x)=\left\{\begin{array}{ccrccccc}
 \frac{3n+3}{2}+i  &    &   \ {\rm if} \ &   \frac{n+9}{2}\leq i \leq n  &  {\rm and}    &   (i\stackrel{4}{\equiv} 0 & {\rm or} & i\stackrel{4}{\equiv} 3)\\
   &    &  &      &  &    &  & \\
 3n+4-2i  &    &   \ {\rm if} \ &   \frac{n+9}{2}\leq i \leq n  &   {\rm and}   &  (i\stackrel{4}{\equiv} 1 & {\rm or} & i\stackrel{4}{\equiv} 2)\\
\end{array}\right.$$
and
$$f(v_{i}y)=\left\{\begin{array}{ccrccccc}
 3n+4-2i  &    &   \ {\rm if} \ &   \frac{n+9}{2}\leq i \leq n  &  {\rm and}    &   (i\stackrel{4}{\equiv} 0 & {\rm or} & i\stackrel{4}{\equiv} 3)\\
   &    &  &      &  &    &  & \\
 \frac{3n+3}{2}+i  &    &   \ {\rm if} \ &   \frac{n+9}{2}\leq i \leq n  &   {\rm and}   &  (i\stackrel{4}{\equiv} 1 & {\rm or} & i\stackrel{4}{\equiv} 2).\\
\end{array}\right.$$

\noindent Since $\frac{n+1}{2}$ is divisible by $4$, the number of vertices in $\left\{v_{i}|\ \frac{n+9}{2}\leq i \leq n\right\}$
is divisible by $4$. Also, $\frac{n+9}{2}\stackrel{4}{\equiv} 0$.
Now, if $\{i,i+1,i+2,i+3\} \subseteq \left\{\frac{n+9}{2},\frac{n+11}{2}, \ldots , n-1,n\right\}$
and $i\stackrel{4}{\equiv}0$, we have
\begin{center}
$\sum\limits_{j=i}^{i+3} f(v_{j}x)=\sum\limits_{j=i}^{i+3} f(v_{j}y)$.
\end{center}
Accordingly, 
\begin{center}
$\sum\limits_{j=\frac{n+9}{2}}^{n} f(v_{j}x)=\sum\limits_{j=\frac{n+9}{2}}^{n} f(v_{j}y)$.
\end{center}
We conclude that
\begin{center}
$f(vx)+\sum\limits_{j=1}^{n} f(v_{j}x)=f(vy)+\sum\limits_{j=1}^{n} f(v_{j}y)$;
\end{center}
and the proof is completed in this case.
\\

\noindent \textbf{Case 2.} The case that $\frac{n+1}{2}\stackrel{4}{\equiv} 2$.

\noindent For each $i$ in $\{1,2,\ldots ,\frac{n+1}{2}\}$, we define $f(v_{i}x)$ and $f(v_{i}y)$ as follows;

$$f(v_{i}x)=\left\{\begin{array}{ccrccccc}
 \frac{5n+3}{2}+i  &    &   \ {\rm if} \ &   1\leq i \leq \frac{n+1}{2}  &  {\rm and}    &   (i\stackrel{4}{\equiv} 1 & {\rm or} & i\stackrel{4}{\equiv} 0)\\
   &    &  &      &  &    &  & \\
 2n+4-2i  &    &   \ {\rm if} \ &   1\leq i \leq \frac{n+1}{2}  &   {\rm and}   &  (i\stackrel{4}{\equiv} 2 & {\rm or} & i\stackrel{4}{\equiv} 3)\\
\end{array}\right.$$
and
$$f(v_{i}y)=\left\{\begin{array}{ccrccccc}
 2n+4-2i  &    &   \ {\rm if} \ &   1\leq i \leq \frac{n+1}{2}  &   {\rm and}   & (i\stackrel{4}{\equiv} 1 & {\rm or} & i\stackrel{4}{\equiv} 0)\\ 
   &    &  &        &   &   &  & \\
   \frac{5n+3}{2}+i  &    &   \ {\rm if} \ &   1\leq i \leq \frac{n+1}{2}  &   {\rm and}   & (i\stackrel{4}{\equiv} 2 & {\rm or} & i\stackrel{4}{\equiv} 3).\\
\end{array}\right.$$

\noindent If $\{i,i+1,i+2,i+3\} \subseteq \left\{1,2, \ldots , \frac{n+1}{2}\right\}$ and 
$i\stackrel{4}{\equiv} 1$, then
\begin{center}
$\sum\limits_{j=i}^{i+3} f(v_{j}x)=\sum\limits_{j=i}^{i+3} f(v_{j}y)$.
\end{center}
Because of $\frac{n+1}{2}\stackrel{4}{\equiv} 2$, we have
\begin{center}
$\sum\limits_{i=1}^{\frac{n+1}{2}} f(v_{i}y)=3+\sum\limits_{i=1}^{\frac{n+1}{2}} f(v_{i}x)$.
\end{center}
By setting the following four assignments
$$\begin{array}{lcr}
 f\left(xv_{\frac{n+3}{2}}\right)=\frac{3n+3}{2}+\frac{n+3}{2},  &   & f(vx)=n+2, \\
   &    &  \\
 f\left(yv_{\frac{n+3}{2}}\right)=
 3n+4-2\left(\frac{n+3}{2}\right),  &    &  f(vy)=n+1,  
\end{array}$$
we obtain
\begin{center}
$f(vx)+\sum\limits_{i=1}^{\frac{n+3}{2}} f(v_{i}x)=f(vy)+\sum\limits_{i=1}^{\frac{n+3}{2}} f(v_{i}y).$
\end{center}
Now, we determine the exact values of
\begin{center}
$f\left(xv_{\frac{n+5}{2}}\right),f\left(yv_{\frac{n+5}{2}}\right),
f\left(xv_{\frac{n+7}{2}}\right),f\left(yv_{\frac{n+7}{2}}\right),
\ldots ,f(xv_{n}),f(yv_{n})$.
\end{center}
Let us regard the following assignments;

$$f(v_{i}x)=\left\{\begin{array}{ccrccccc}
 \frac{3n+3}{2}+i  &    &   \ {\rm if} \ &   \frac{n+5}{2}\leq i \leq n  &  {\rm and}    &   (i\stackrel{4}{\equiv} 0 & {\rm or} & i\stackrel{4}{\equiv} 3)\\
   &    &  &      &  &    &  & \\
 3n+4-2i  &    &   \ {\rm if} \ &   \frac{n+5}{2}\leq i \leq n  &   {\rm and}   &  (i\stackrel{4}{\equiv} 1 & {\rm or} & i\stackrel{4}{\equiv} 2)\\
\end{array}\right.$$
and
$$f(v_{i}y)=\left\{\begin{array}{ccrccccc}
 3n+4-2i  &    &   \ {\rm if} \ &   \frac{n+5}{2}\leq i \leq n  &  {\rm and}    &   (i\stackrel{4}{\equiv} 0 & {\rm or} & i\stackrel{4}{\equiv} 3)\\
   &    &  &      &  &    &  & \\
 \frac{3n+3}{2}+i  &    &   \ {\rm if} \ &   \frac{n+5}{2}\leq i \leq n  &   {\rm and}   &  (i\stackrel{4}{\equiv} 1 & {\rm or} & i\stackrel{4}{\equiv} 2).\\
\end{array}\right.$$
If $i\stackrel{4}{\equiv}0$ and 
$\{i,i+1,i+2,i+3\} \subseteq \left\{\frac{n+5}{2},\frac{n+7}{2}, \ldots ,n\right\}$,
then
\begin{center}
$\sum\limits_{j=i}^{i+3} f(v_{j}x)=\sum\limits_{j=i}^{i+3} f(v_{j}y)$.
\end{center}
Thus, since $\frac{n+5}{2}\stackrel{4}{\equiv} 0$ and
the number of vertices in $\left\{v_{\frac{n+5}{2}},v_{\frac{n+7}{2}}, \ldots ,v_{n}\right\}$
is divisible by $4$, we obtain that
\begin{center}
$\sum\limits_{j=\frac{n+5}{2}}^{n} f(v_{j}x)=\sum\limits_{j=\frac{n+5}{2}}^{n} f(v_{j}y)$.
\end{center}
Accordingly,
\begin{center}
$f(vx)+\sum\limits_{j=1}^{n} f(v_{j}x)=f(vy)+\sum\limits_{j=1}^{n} f(v_{j}y)$;
\end{center}
which is desired in this case.
\\

\noindent \textbf{Case 3.} The case that $\frac{n+1}{2}$ is odd and $\frac{n+3}{2}$ is
divisible by $3$.

\noindent In this case, $\frac{n-1}{2}$ is even. Also, since $\frac{n+3}{2}$ is
divisible by $3$, both of $\frac{n}{3}$ and $\frac{n-3}{6}$ are integers. We define
$$\begin{array}{lcr}
 f\left(v_{1}x\right)=2n+4-2,  &   & f(vx)=n+1, \\
   &    &  \\
 f\left(v_{1}y\right)=\frac{5n+3}{2}+1,  &    &  f(vy)=n+2.  
\end{array}$$
For each $i$ with $2\leq i\leq \frac{n}{3}$, set $f(v_{i}x)$ and $f(v_{i}y)$ as the following;

$$f(v_{i}x)=\left\{\begin{array}{ccrl}
 \frac{5n+3}{2}+i  & \ \  &   \ {\rm if} \ &   2\leq i \leq \frac{n}{3}\  {\rm and}\  i\  {\rm is\ even} \\
   &    &   &  \\
 2n+4-2i  &    \ \  &   \ {\rm if} \ &     2\leq i \leq \frac{n}{3}  \  {\rm and}\  i\  {\rm is\ odd} \\
\end{array}\right.$$
and
$$f(v_{i}y)=\left\{\begin{array}{ccrl}
 2n+4-2i  &    \ \  &   \ {\rm if} \ &   2\leq i \leq \frac{n}{3}  \  {\rm and}\  i\  {\rm is\ even} \\ 
   &    &    & \\
   \frac{5n+3}{2}+i  &    \ \  &   \ {\rm if} \ &     2\leq i \leq \frac{n}{3}  \  {\rm and}\  i\  {\rm is\ odd.} \\
\end{array}\right.$$
\\
It is obvious that if $i$ is an even integer with $2\leq i \leq \frac{n}{3}$,
then
\begin{center}
$f(yv_{i})+f(yv_{i+1})=f(xv_{i})+f(xv_{i+1})+3$.
\end{center}
So,
\begin{center}
$\sum\limits _{i=2}^{\frac{n}{3}} f(yv_{i}) =\frac{n-3}{2} + \sum\limits _{i=2}^{\frac{n}{3}} f(xv_{i}).$
\end{center}
Now, for each $i$ with
$\frac{n+3}{3} \leq i \leq \frac{n+1}{2}$, define $f(v_{i}x)$ and $f(v_{i}y)$ as the following;

$$f(v_{i}x)=\left\{\begin{array}{ccrc}
 2n+4-2i  & \ \  &   \ {\rm if} \ &   \frac{n+3}{3}\leq i \leq \frac{n+1}{2}\  {\rm and}\  i\  {\rm is\ even} \\
   &    &    & \\
 \frac{5n+3}{2}+i  &    \ \  &   \ {\rm if} \ &     \frac{n+3}{3}\leq i \leq \frac{n+1}{2}  \  {\rm and}\  i\  {\rm is\ odd} \\
\end{array}\right.$$
and
$$f(v_{i}y)=\left\{\begin{array}{ccrc}
 \frac{5n+3}{2}+i  &    \ \  &   \ {\rm if} \ &   \frac{n+3}{3}\leq i \leq \frac{n+1}{2}  \  {\rm and}\  i\  {\rm is\ even} \\ 
   &    &   &  \\
   2n+4-2i  &    \ \  &   \ {\rm if} \ &     \frac{n+3}{3}\leq i \leq \frac{n+1}{2}  \  {\rm and}\  i\  {\rm is\ odd.} \\
\end{array}\right.$$
\\
If $i$ is an even integer with $\frac{n+3}{3}\leq i \leq \frac{n+1}{2}$,
then we have
\begin{center}
$f(xv_{i})+f(xv_{i+1})=f(yv_{i})+f(yv_{i+1})+3$.
\end{center}
Therefore,
\begin{center}
$\sum\limits _{i=\frac{n+3}{3}}^{\frac{n+1}{2}} f(xv_{i}) =\frac{n+3}{4} + \sum\limits _{i=\frac{n+3}{3}}^{\frac{n+1}{2}} f(yv_{i}).$
\end{center}
Finally, let us regard the following assignments for $f(v_{i}x)$ and $f(v_{i}y)$
when $i$ is an integer with $\frac{n+3}{2}\leq i \leq n$;

$$f(v_{i}x)=\left\{\begin{array}{ccrc}
 3n+4-2i  & \ \  &   \ {\rm if} \ &   \frac{n+3}{2}\leq i \leq n\  {\rm and}\  i\  {\rm is\ even} \\
   &    &    & \\
 \frac{3n+3}{2}+i  &    \ \  &   \ {\rm if} \ &     \frac{n+3}{2}\leq i \leq n  \  {\rm and}\  i\  {\rm is\ odd} \\
\end{array}\right.$$
and
$$f(v_{i}y)=\left\{\begin{array}{ccrc}
 \frac{3n+3}{2}+i  &    \ \  &   \ {\rm if} \ &   \frac{n+3}{2}\leq i \leq n  \  {\rm and}\  i\  {\rm is\ even} \\ 
   &    &    & \\
   3n+4-2i  &    \ \  &   \ {\rm if} \ &     \frac{n+3}{2}\leq i \leq n  \  {\rm and}\  i\  {\rm is\ odd.} \\
\end{array}\right.$$
\\
Again, for each even integer $i$ with $\frac{n+3}{2}\leq i \leq n$
we have
\begin{center}
$f(xv_{i})+f(xv_{i+1})=f(yv_{i})+f(yv_{i+1})+3$.
\end{center}
Thus,
\begin{center}
$\sum\limits _{i=\frac{n+3}{2}}^{n} f(xv_{i}) =\frac{3(n-1)}{4} + \sum\limits _{i=\frac{n+3}{2}}^{n} f(yv_{i}).$
\end{center}
We conclude that
\begin{center}
$f(vx)+\sum\limits_{i=1}^{n} f(v_{i}x)=f(vy)+\sum\limits_{i=1}^{n} f(v_{i}y)$;
\end{center}
which completes the proof in this case.
\\

\noindent \textbf{Case 4.} The case that $\frac{n+1}{2}$ is odd and $\frac{n-1}{2}$ is
divisible by $3$.

\noindent In this case, we define
$$\begin{array}{lcr}
 f\left(v_{1}x\right)=2n+4-2,  &   & f(vx)=n+2, \\
   &    &  \\
 f\left(v_{1}y\right)=\frac{5n+3}{2}+1,  &    &  f(vy)=n+1.  
\end{array}$$
For each $i$ with $2\leq i\leq \frac{n+2}{3}$, put $f(v_{i}x)$ and $f(v_{i}y)$ as the following;

$$f(v_{i}x)=\left\{\begin{array}{ccrl}
 \frac{5n+3}{2}+i  & \ \  &   \ {\rm if} \ &   2\leq i \leq \frac{n+2}{3}\  {\rm and}\  i\  {\rm is\ even} \\
   &    &    & \\
 2n+4-2i  &    \ \  &   \ {\rm if} \ &     2\leq i \leq \frac{n+2}{3}  \  {\rm and}\  i\  {\rm is\ odd} \\
\end{array}\right.$$
and
$$f(v_{i}y)=\left\{\begin{array}{ccrl}
 2n+4-2i  &    \ \  &   \ {\rm if} \ &   2\leq i \leq \frac{n+2}{3}  \  {\rm and}\  i\  {\rm is\ even} \\ 
   &    &   &  \\
   \frac{5n+3}{2}+i  &    \ \  &   \ {\rm if} \ &     2\leq i \leq \frac{n+2}{3}  \  {\rm and}\  i\  {\rm is\ odd.} \\
\end{array}\right.$$
\\
For each even integer $i$ with $2\leq i \leq \frac{n+2}{3}$,
the following equality holds;
\begin{center}
$f(yv_{i})+f(yv_{i+1})=f(xv_{i})+f(xv_{i+1})+3$.
\end{center}
This implies that
\begin{center}
$\sum\limits _{i=2}^{\frac{n+2}{3}} f(yv_{i}) =\frac{n-1}{2} + \sum\limits _{i=2}^{\frac{n+2}{3}} f(xv_{i}).$
\end{center}
For each $i$ with
$\frac{n+5}{3} \leq i \leq \frac{n+1}{2}$, we define $f(v_{i}x)$ and $f(v_{i}y)$ as follows;

$$f(v_{i}x)=\left\{\begin{array}{ccrc}
 2n+4-2i  & \ \  &   \ {\rm if} \ &   \frac{n+5}{3} \leq i \leq \frac{n+1}{2}\  {\rm and}\  i\  {\rm is\ even} \\
   &    &    & \\
 \frac{5n+3}{2}+i  &    \ \  &   \ {\rm if} \ &     \frac{n+5}{3} \leq i \leq \frac{n+1}{2}  \  {\rm and}\  i\  {\rm is\ odd} \\
\end{array}\right.$$
and
$$f(v_{i}y)=\left\{\begin{array}{ccrc}
 \frac{5n+3}{2}+i  &    \ \  &   \ {\rm if} \ &   \frac{n+5}{3} \leq i \leq \frac{n+1}{2}  \  {\rm and}\  i\  {\rm is\ even} \\ 
   &    &    & \\
   2n+4-2i  &    \ \  &   \ {\rm if} \ &     \frac{n+5}{3} \leq i \leq \frac{n+1}{2}  \  {\rm and}\  i\  {\rm is\ odd.} \\
\end{array}\right.$$
Now, for each even integer $i$ that $\frac{n+5}{3} \leq i \leq \frac{n+1}{2}$
we have
\begin{center}
$f(xv_{i})+f(xv_{i+1})=f(yv_{i})+f(yv_{i+1})+3$.
\end{center}
So, we obtain
\begin{center}
$\sum\limits _{i=\frac{n+5}{3}}^{\frac{n+1}{2}} f(xv_{i}) =\frac{n-1}{4} + \sum\limits _{i=\frac{n+5}{3}}^{\frac{n+1}{2}} f(yv_{i}).$
\end{center}
Now, it is time to determine $f(v_{i}x)$ and $f(v_{i}y)$
for those integers $i$ that $\frac{n+3}{2}\leq i \leq n$.
Let us assign

$$f(v_{i}x)=\left\{\begin{array}{ccrc}
 3n+4-2i  & \ \  &   \ {\rm if} \ &   \frac{n+3}{2}\leq i \leq n\  {\rm and}\  i\  {\rm is\ even} \\
   &    &    & \\
 \frac{3n+3}{2}+i  &    \ \  &   \ {\rm if} \ &     \frac{n+3}{2}\leq i \leq n  \  {\rm and}\  i\  {\rm is\ odd} \\
\end{array}\right.$$
and
$$f(v_{i}y)=\left\{\begin{array}{ccrc}
 \frac{3n+3}{2}+i  &    \ \  &   \ {\rm if} \ &   \frac{n+3}{2}\leq i \leq n  \  {\rm and}\  i\  {\rm is\ even} \\ 
   &    &    & \\
   3n+4-2i  &    \ \  &   \ {\rm if} \ &     \frac{n+3}{2}\leq i \leq n  \  {\rm and}\  i\  {\rm is\ odd.} \\
\end{array}\right.$$
\\
Since the equality
$f(xv_{i})+f(xv_{i+1})=f(yv_{i})+f(yv_{i+1})+3$
holds
for each even integer $i$ that $\frac{n+3}{2}\leq i \leq n$,
we have
\begin{center}
$\sum\limits _{i=\frac{n+3}{2}}^{n} f(xv_{i}) =\frac{3(n-1)}{4} + \sum\limits _{i=\frac{n+3}{2}}^{n} f(yv_{i}).$
\end{center}
Accordingly,
\begin{center}
$f(vx)+\sum\limits_{i=1}^{n} f(v_{i}x)=f(vy)+\sum\limits_{i=1}^{n} f(v_{i}y)$;
\end{center}
and therefore, the proof is completed in the final case.
}
\end{proof}

\end{document}